\documentclass{gtart}
\usepackage{amsmath, amssymb, verbatim, psfrag, color,amscd} 

\newtheorem{thm}{Theorem}[section]

\newtheorem{cor}[thm]{Corollary}
\newtheorem{lem}[thm]{Lemma}

\newtheorem{prop}[thm]{Proposition}
\newtheorem{defn}[thm]{Definition}

\newtheorem{rem}[thm]{Remark}

\numberwithin{equation}{section}
\newcommand{\HHFm}{\mathbf{HF^-}}

\newcommand\Field{\mathbb F}

\newcommand\Dual{\mathcal D}
\newcommand\Duality\Dual

\newcommand\x{\mathbf x}

\newcommand\ModSphere{\ModFlow\left({\mathbb S}\longrightarrow 
\Sym^{g-1}(\Sigma_{1})\times \Sym^2(\Sigma_{2})\right)}
\newcommand\ModSpheres\ModSphere

\newcommand\HFm{{\rm {HF}}^-}
\newcommand\CFinf{{{\rm {CF}}}^\infty}

\newcommand\UnparModSp{\widehat \ModSp}
\newcommand\UnparModFlow\UnparModSp
\newcommand\Mod\ModSp

\newcommand\ModMaps{\mathcal M}
\newcommand\ModSp\ModMaps

\newcommand{\bfn}{{\mathbb {N}}}
\newcommand{\bfq}{{\mathbb {Q}}}

\newcommand{\Vertices }{{\rm {Vert}}}
\newcommand{\uuinv}{U^{-1}, U}

\newcommand{\s}{\mathbf s} 
 
\renewcommand{\u}{\mathbf u} 
\renewcommand{\t}{\mathbf t}
\newcommand{\DD}{\mathfrak D}

 \newcommand{\Z}{\mathbb Z} \newcommand{\N}{\mathbb
  N} \newcommand{\Q}{\mathbb Q} \newcommand{\R}{\mathbb R}

\newcommand{\CFm}{{\rm {CF}} ^-}

\newcommand\alphak{\mbox{\boldmath$\alpha$}}
\newcommand\betak{\mbox{\boldmath$\beta$}}

\newcommand\CFinfComb{{\mathbb{CF}}^{\infty}}
\newcommand\HFinfComb{{\mathbb{HF}}^{\infty}}
\newcommand\CFmComb{{\mathbb{CF}}^{-}}

\newcommand\HFmComb{{\mathbb{HF}}^{-}}
\newcommand\HF{\rm {HF}}

\begin{document}

\title{Knot lattice homology in $L$-spaces}

\author{Peter Ozsv\'ath}
\address{Department of Mathematics, Princeton University\\
Princeton, NJ, 08544}
\email{petero@math.princeton.edu}

\author{Andr\'{a}s I. Stipsicz}
\address{R{\'e}nyi Institute of Mathematics\\
Budapest, Hungary and \\
Institute for Advanced Study, Princeton, NJ, 08540}
\email{stipsicz@math-inst.hu}

\author{Zolt\'an Szab\'o}
\address{Department of Mathematics, Princeton University\\
Princeton, NJ, 08544}
\email{szabo@math.princeton.edu}


\begin{abstract}
  We show that the knot lattice homology of a knot in an $L$-space is
  equivalent to the knot Floer homology of the same knot (viewed these
  invariants as filtered chain complexes over the polynomial ring
  $\Z/2\Z [U]$).  Suppose that $G$ is a negative definite plumbing
  tree which contains a vertex $w$ such that $G-w$ is a union of
  rational graphs. Using the identification of knot homologies we show
  that for such graphs the lattice homology $\HFmComb (G)$ is
  isomorphic to the Heegaard Floer homology $\HFm (Y_G)$ of the
  corresponding rational homology sphere $Y_G$.
\end{abstract}

\maketitle

\section{Introduction}
\label{sec:intro}
Suppose that $G$ is a negative definite plumbing tree determining the
rational homology sphere $Y_G$. \emph{Heegaard Floer homology} (as it
is introduced in \cite{OSzF1, OSzF2}) associates to $Y_G$ a finitely
generated $\Field [U]$-module $\HFm (Y_G)$ (where $\Field$ denotes the
field $\Z/2\Z$ of two element), which splits according to spin$^c$
structures of $Y_G$ and also admits an absolute Maslov grading
\cite{OSzabs}.  This $\Field [U]$-module is defined as the homology of
a chain complex $\CFm (Y)$ associated to a Heegaard diagram of the
3-manifold $Y_G$.  The definition of the chain complex, in turn,
involves Floer theoretic constructions for symplectic manifolds and
Lagrangians in them associated to the Heegaard diagram of $Y_G$. In
particular, the boundary map of the chain complex counts certain pseudo-holomorphic disks in
high-dimensional symplectic manifolds.

In \cite{lattice} N\'emethi defined an invariant of negative definite
plumbing graphs which is computed as the homology of a chain complex
associated to the graph in a purely combinatorial manner.  As it was
shown in \cite{lattice}, the resulting \emph{lattice homology}
$\HFmComb (G)$ captures interesting information about singularities
with resolution graph $G$.  Heegaard Floer homology and lattice homology
share a number of common properties. The algebraic structures of the
two theories are similar, both satisfy certain exact triangles
(\cite{Josh, latticetriangle, OSzF2}), and for certain plumbings it is
easy to verify that the two homologies are, in fact, isomorphic
$\Field [U]$-modules \cite{lattice}.

As a further common property, both theories admit refinements for knots in
3-manifolds \cite{latticeknot, OSzknot}.  This refinement on the Heegaard
Floer side is defined for a pair $(Y,K)$ of a knot $K\subset Y$ in the
3-manifold $Y$, and admits the shape of a filtration $A$ on the chain complex
$\CFm (Y)$ computing the Heegaard Floer homology $\HFm (Y)$. (At least this is the
case for 3-manifolds with $b_1(Y)=0$.) In lattice homology, the knots are of a
special type.  They are specified by a tree $\Gamma _{v_0}$ with a
distingushed vertex $v_0$ such that $G=\Gamma _{v_0}-v_0$ is a negative
definite plumbing tree.  Once again, the invariant is a pair $(\CFmComb (G),
A)$, where $\CFmComb (G)$ is the chain complex associated to the background
plumbing graph $G$ computing the lattice homology, while $A$ is a filtration
on this $\Field [U]$-module.  (For more about these refinements see
\cite{latticeknot, OSzknot}, the discussion below and also the beginning
paragraphs of Sections~\ref{sec:hf} and \ref{sec:lattice}.)

It is  natural to expect that the two homology theories are
isomorphic for all negative definite plumbing trees of spheres.
Indeed, such isomorphisms have been already verified for a number of
families of graphs, cf. \cite{lattice, latticeknot, latticess}.  Such
an isomorphism provides a convenient description of the important
3-manifold invariants $\HFm (Y_G)$ --- at least for these special
3-manifolds. 

In the present paper we show that the filtered chain complexes in the
two theories for specific graphs $\Gamma _{v_0}$ are filtered chain
homotopic, and resting on this result we extend the family of plumbing
trees/forests for which the two theories produce isomorphic homology
groups.  The types of graphs for which these ideas apply will be
specified below.

In order to state our results, we need to consider a few definitions,
notions and constructions. Suppose that $\Gamma _{v_0}$ is a tree with
a distinguished vertex $v_0$, and all further vertices of $\Gamma
_{v_0}$ admit some integral framing. As before, let $G$ denote the
plumbing graph $\Gamma _{v_0}-v_0$.  As a plumbing graph, $G$ actually
gives rise to a surgery presentation of the plumbed 3-manifold $Y_G$:
replace each vertex by an unknot, link two unknots if and only if the
corresponding vertices are connected by an egde in $G$, and decorate
the unknots with the integers attached to the corresponding vertices.
After performing the surgeries on the unknots corresponding to the
vertices of $G$ (with the given framing), we get the 3-manifold $Y_G$
and the vertex $v_0$ defines a knot $K=K_{v_0}$ in $Y_G$. The
plumbing graph $G$ also determines a simply connected 4-manifold $X_G$ we
get by plumbing disk bundles over spheres together according to $G$,
and $K\subset \partial X_G=Y_G$.

Suppose that the plumbing tree $G$ is negative definite. Denote the
homology class in the plumbed 4-manifold corresponding to the vertex
$v_i$ of the graph by $E_i$. According to \cite{Artin} there is a
nonzero element $Z=\sum _i n_i\cdot E_i\in H_2 (X_G; \Z )$ with the
property that $n_i$ are all nonnegative integers, for every $i$ we have
$Z\cdot E_i\leq 0$, and for any other $Z'=\sum _i n_i 'E_i$ with the
same properties $n_i\leq n_i'$ holds for every $i$.  (The dot product
is computed using the intersection matrix $M_G$ on $H_2(X_G; \Z )$.)
\begin{defn}
The plumbing graph $G$ is \emph{rational} if for the class
$Z=\sum _i n_i E_i$ discussed above, the equality
\[
(\sum _i n_i E_i)^2=2\sum _i n_i+\sum _i n_iE_i^2 -2
\]
holds. (This condition is equivalent to requiring that the geometric genus
$p(Z)=\frac{1}{2}(Z^2+K\cdot Z)+1$ of the class $Z$ vanishes.)
\end{defn}

For a simple algorithm deciding whether a graph is rational, see
Remark~\ref{rem:teszt}.  (For further, more analytic characterizations
of rationality, see for example \cite{nemethi}.)  Following \cite{nemethi-ar} we
say that the negative definite plumbing tree $G$ is \emph{almost
  rational} if there is a vertex $w$ in $G$ with the property that by
decreasing the framing on $w$ sufficiently we get a rational graph.
The main result of the paper is the following theorem:

\begin{thm} \label{thm:main} Suppose that $\Gamma _{v_0}$ is a
  tree/forest with a distinguished vertex $v_0$, $G=\Gamma _{v_0}-v_0$
  is negative definite, and each of its components is rational. Then
  the filtered chain complex $(\CFmComb (G), A)$ of $\Gamma _{v_0}$ in
  lattice homology is filtered chain homotopy equivalent to the
  filtered complex $(\CFm (Y_G), A)$ of the pair $(Y_G, K_{v_0})$ in
  Heegaard Floer homology.
\end{thm}

As a simple corollary we get the following: 
\begin{cor}
The knot lattice homology of a knot 
in $S^3$ is equal to the knot Floer homology of the same knot. \qed
\end{cor}
\begin{rem}
  The knots produced by the above construction are quite special. It can
  be shown that if $v_0$ is a leaf and $G=\Gamma _{v_0}-v_0$
  represents $S^3$, then $G$ can be sequentially blown down, and
  therefore the knot represented by the vertex $v_0$ is an iterated
  torus knot.  If $v_0$ is not a leaf, then $\Gamma _{v_0}$ can be
  presented as the connected sum of other graphs with the
  distinguished vertex being a leaf, and hence the knot corresponding
  to $v_0$ in the general case is the connected sum of certain
  iterated torus knots.
\end{rem}

As a further corollary we verify the following result:

\begin{thm}\label{thm:egycsucs}
Suppose that $G$ is a negative definite plumbing tree with the property that
it admits a vertex $w$ such that all the components of $G-w$ are rational 
graphs. 
Then the lattice homology $\HFmComb (G)$ of $G$ is isomorphic (as a 
Maslov-graded $\Field [U]$-module) to the Heegaard Floer homology 
$\HFm (Y_G)$ of the 3-manifold $Y_G$ defined by $G$.
\end{thm}

Notice that every almost rational graph (in the sense of
\cite{lattice}), and so in particular evey graph with one bad vertex
is considered by the above theorem.  For such graphs the stated
isomorphism was already proved by N\'emethi \cite{lattice} --- indeed,
the result for almost rational graphs will be used in the proof of our
theorem.  Theorem~\ref{thm:egycsucs}, however, applies to many more
graphs, among which we can find type-$k$ graphs for arbitrary
$k$. (Recall~\cite{latticess} that a plumbing graph is said to be {\em
  of type-$k$} if we can find $k$ vertices $\{v_{i_1},\dots,v_{i_k}\}$
on which we can change the framings $\{ m_{i_1}, \ldots , m_{i_k}\}$
in such a way that the resulting graph is rational.)  For similar
results see also \cite{lattice, s3csomok, latticeknot}.

{\bf Acknowledgements}: PSO was supported by NSF grant number
DMS-0804121.  AS was supported by OTKA NK81203, by ERC Grant LDTBud,
by \emph{Lend\"ulet program} and by the Institute for Advanced Study.
ZSz was supported by NSF grants DMS-0603940, DMS-0704053 and
DMS-1006006.

\section{Heegaard Floer homology}
\label{sec:hf}

We start our discussion by briefly recalling the relevant notions and
constructions of Heegaard Floer homology. For a thorough discussion
the reader is advised to consult \cite{OSzF1, OSzF2, OSzknot,
  OSzlinks}.

Suppose that $K\subset Y$ is a given knot, and for simplicity assume that $Y$
is a rational homology sphere, i.e. $b_1(Y)=0$. Consider a doubly pointed
Heegaard diagram $\DD _K = (\Sigma , \alphak , \betak , w, z )$ compatible
with $(Y,K)$.  Let ${\mathbb {T}}_{\alpha}$ (and ${\mathbb {T}}_{\beta}$,
respectively) denote product tori $\times _{i=1}^g \alpha _i$ (and $\times
_{i=1} ^g \beta _i$, resp.) in the $g$-fold symmetric product ${\rm
  {Sym}}^g(\Sigma )$ of the genus-$g$ surface $\Sigma $.  The diagram $\DD =
(\Sigma , \alphak , \betak , w)$ (and a spin$^c$ structure $\s$) gives rise to
a chain complex $\CFm (\DD , \s )$ over $\Field [U]$, which is naturally a
subcomplex of $\CFinf (\DD , \s )=\CFm (\DD , \s )\otimes _{\Field [U]}\Field
[\uuinv]$. For both cases the generators of the modules are given by the
intersection points ${\mathbb {T}}_{\alpha }\cap {\mathbb {T}}_{\beta }$ in
${\rm {Sym}}^g (\Sigma )$, and the boundary map counts certain
pseudo-holomorphic disks in this symmetric product (after fixing appropriate
almost complex structure on it).  Since $Y$ is a rational homology sphere (and
so, in particular, all spin$^c$ structures are torsion), the chain complexes
come with absolute Maslov gradings \cite{OSzabs}, taking values in $\bfq$. The
boundary map of the complex decreases this grading by one (and therefore the
Maslov grading descends to the homologies), while multiplication by $U$
decreases the Maslov grading by two.  Indeed, since $Y$ is a rational homology
sphere, for a fixed spin$^c$ structure $\s$ the $\Field [U]$-module $\HFm (Y,
\s )=H_*(\CFm (\DD , \s ), \partial ^-)$ has the algebraic structure of a sum
of $\Field [U]$ with a finitely generated $U$-torsion module.  The Maslov
grading of the generator of the free part is called the \emph{$d$-invariant}
$d(Y, \s )$ of the underlying spin$^c$ 3-manifold $(Y, \s )$. Our
present grading conventions on
$\HFm$ are different from those from~\cite{OSzabs}: according to
our present conventions, the generator of $\HFm(S^3)\cong\Field[U]$ has Maslov
grading equal to $0$, rather than $-2$. However, the $d$ invariants are the same.

We use the further basepoint $z$ of $\DD _K$ to equip the module $\CFm
(\DD , \s )$ by the Alexander grading $A$, inducing the
\emph{Alexander filtration} on $\CFm (\DD , \s )$.  At first glance
the theory only provides a relative Alexander grading on $\CFm (\DD ,
\s )$ for a fixed spin$^c$ structure $\s$. For a rational homology
sphere, however, these gradings can be coordinated for the various
spin$^c$ structures, and by requiring a global symmetry we get an
absolute lift, which induces the filtration we considered above.  The
Alexander grading in general is not assumed to be integer
valued. Nevertheless, for every spin$^c$ structure $\s$ there is a
rational number $i_{\s}\in [0,1)$ with the property that for any
  generator $\x$ with spin$^c$ structure $\s$ the Alexander grading
  $A(\x )$ is congruent to $i _{\s }$ mod 1. Multiplication by $U$
  decreases the Alexander grading by 1.

Notice that the filtered chain complex $(\CFm (\DD , \s), A)$
naturally gives rise to the doubly filtered chain complex $(\CFinf
(\DD , \s ), j , A)$ where $(\CFinf (\DD , \s ), j , A)$ is naturally
an $\Field [\uuinv ]$-module and the filtration $j$ measures the negative 
of the
exponent of $U$. In particular
\[
\{ x\in \CFinf (\DD , \s )\mid j(x )\leq 0 \}
\]
is naturally isomorphic to $\CFm (\DD , \s )$.

The double filtration allows us to consider the following 
subcomplexes:
\[
B(\s )=\{ x\in \CFinf (\DD , \s )\mid j(x)\leq 0\},
\]
\[
C_i(\s ) =\{ x\in \CFinf (\DD , \s )\mid A(x)\leq i\} ,
\]
\[
A_i(\s )=C_i (\s )\cap B(\s )=\{ x\in \CFinf (\DD , \s )\mid
j(x)\leq 0, \ A(x)\leq i\}.
\]
There are natural embeddings $\psi =\psi _i(\s) \colon A_i(\s) \to
A_{i+1}(\s), v=v_i(\s )\colon A_i(\s )\to B(\s )$ and $h=h_i(\s
)\colon A_i (\s )\to C_i (\s )$. The map $\phi = \phi _{i+1}(\s
)\colon A_{i+1}(\s )\to A_i (\s )$ is defined by composing the
multiplication-by-$U$ map with the natural embedding. 
(The rational number $i$ above is congruent to $i_{\s}\in [0,1)$.)

The doubly filtered chain complex, and in particular the subcomplexes
defined (and the maps considered) above can be conveniently used to
describe a chain complex computing Heegaard Floer homology for any
surgery along $K$ \cite{OSzint, OSzrac}. (In the following we will
concentrate exclusively on integral surgeries.)

Before stating the theorem connecting the filtered chain complexes and
Heegaard Floer homolgies of surgeries, we introduce the following
notation. Suppose that $\s$ is a given spin$^c$ structure on $Y$.  Let
$\s _{[K]}$ denote the spin$^c$ structure we get by twisting $\s$ with
the Poincare dual of $[K]$. In particular, $c_1(\s _{[K]})= c_1 (\s
)+2PD[K]$. (If $K$ is null-homologous, then $\s=\s _{[K]}$.)  More
generally, for any integer $n\in \Z$ the spin$^c$ structure
$\s_{n[K]}$ is given by applying the previous construction $\vert n
\vert$-times, using $[K]$ for positive $n$ and $-[K]$ for negative
$n$.

For the following theorem, it is useful to use a completed version of
Heegaard Floer homology, $\HHFm(Y)$. This can be defined as the
homology of $\CFm(Y)\otimes _{\Field [U]}\Field[[U]]$, where
$\Field[[U]]$ denotes the ring of formal power series in $U$. There
are corresponding completions ${\mathbf{B}}(\s)$,
${\mathbf{A}}_i(\s)$, and ${\mathbf{C}}_i(\s)$ of $B(\s)$, $A_i(\s)$,
and $C_i(\s)$ respectively.  The maps $\phi, \psi, v$ and $h$
naturally extend to these completions.

\begin{thm}(\cite{OSzint}) \label{thm:oszint} 
Let $Y_p(K)$ denote the result of some integer surgery along $K$ and
fix a spin$^c$ structure $\t$ on $Y_p(K)$. Suppose furthermore that
$b_1(Y_p(K))=0$. Then there are rational numbers $i$ and $\alpha$, a
spin$^c$ structure $\s \in Spin^c (Y)$ (depending on $p$ and $\t$) and
chain maps $N_n\colon {\mathbf{C}}_{i+\alpha \cdot n}(\s _{n[K]})\to {\mathbf{B}}(\s
_{(n+1)[K]})$ such that the chain complexes
\[
{\mathbb {A}}_{\t }=\prod_{n=-\infty }^{\infty} {\mathbf{A}}_{i+\alpha \cdot n}(\s _{n[K]})
\]
\[
{\mathbb {B}}_{\t }=\prod _{n=-\infty }^{\infty} {\mathbf{B}} (\s _{n[K]})
\]
and the map $f\colon {\mathbb {A}}_{\t } \to {\mathbb {B}}_{\t }$
defined on the component ${\mathbf{A}}_{i+\alpha \cdot n}(\s _{n[K]})$ as
$v_{i+\alpha n}+N_n\circ h_{i+\alpha n}$ provide a mapping cone
$MC({\mathbb {A}}_{\t }, {\mathbb {B}}_{\t }, f)$ with the property
that its homology is isomorphic to $\HHFm (Y_p(K), \t )$. \qed
\end{thm}

In the above theorem one needs to use the completed version of the
theory once $p>0$ is allowed. (In fact, we will use this result for
sufficiently large positive integer $p$.) On the other hand, as our
next lemma shows, for torsion spin$^c$ structures the two theories
(before and after the above completion) determine each other. Indeed,
let $C$ be a free chain complex of free chain complex
over $\Field[U ]$ with finitely generated homology. A {\em Maslov grading} on $C$ is a grading
with values in $\Q$ with the property that the differential drops the
grading by $1$ and multiplication by $U$ drops it by $2$.  The
following is a standard algebraic fact:

\begin{lem}
  \label{lem:CompletionLemma}
  Let $C$ be a free, Maslov-graded chain complex over $\Field[U]$, with
  finitely generated homology.  Let $C[[U]]$ denote its completion, thought of
  as a chain complex over $\Field[[U]]$.  Then, the homology of $C$ and the
  homology of $C[[U]]$ contain the same information; in particular, if $C$ and
  $C'$ are two finitely-generated, Maslov-graded chain complexes and
  $H_*(C[[U]])\cong H_*(C'[[U]])$, then $H_*(C)\cong H_*(C')$.
\end{lem}

\begin{proof}
  By standard homological algebra, $H_*(C)$ is a direct sum of
  finitely many cyclic modules, each of which inherits a (rational)
  Maslov grading. It follows that $H_*(C)$ has the following form:
  $$H_*(C)\cong \Field[U]^{n_0}\oplus \bigoplus_{i=1}^k
  (\Field[U]/U^{n_i}).$$
  But in this case,
  \[
  H_*(C[[U]])\cong \Field[[U]]^{n_0}\oplus \bigoplus_{i=1}^k
  (\Field[U]/U^{n_i}),
  \] 
  concluding the argument.
\end{proof}

Recall that the rational homology sphere $Y$ is called an 
\emph{$L$-space} if $\HFm (Y, \s )$ is isomorphic to 
$\Field [U]$ for every spin$^c$ structure $\s \in Spin^c(Y)$.
With a slight extension of the results of \cite{OSzlens}
(where surgeries along certain knots in $S^3$ were examined) we get
the following:

\begin{thm}\label{thm:aibidi}
Suppose that $Y$ is an $L$-space, and $K\subset Y$ is a knot 
such that there is a sufficiently large positive integer $p$ with the
property that $Y_p(K)$ ($p$-surgery on $K$) is also an $L$-space. Then,
\begin{enumerate}
\item $H_*(B(\s ))=H_* (C_i(\s ))=H_*(A_i(\s ))=\Field [U] $ for every
  spin$^c$ structure $\s $ on $Y$.
\item The map $\psi _i(\s)_*$ induced by the embedding of $A_i (\s )$
  into $A_{i+1}(\s )$ is multiplication by $U^{d_i(\s )}$ with $d_i(\s
  )\in \{ 0, 1\}$.
\item The maps $v_i (\s )_*$ and $h_i(\s)_*$ (induced by the maps
  $v_i(\s )$ and $h_i (\s )$ on homology) are multiplications by
  $U^{a_i(\s)}$ and $U^{b_i(\s)}$ for some nonnegative integers
  $a_i(\s ), b_i(\s)$, respectively.
\item If $d_i(\s )=0$ then $a_i(\s)=a_{i+1}(\s)$ and $b_i(\s
  )=b_{i+1}(\s)-1$, while if $d_i(\s )=1$ then $a_i(\s)=a_{i+1}(\s)+1$
  and $b_i(\s )=b_{i+1}(\s )$.
\item Finally, for all $i,j$ large enough $a_{i}(\s )=0$ and
  $b_{-j}(\s )=0$.
\end{enumerate}
\end{thm}
\begin{proof}
  Since $H_*(B(\s ))$ is isomorphic to $\HFm (Y, \s)$ and $C_i(\s )$
  is chain homotopic to $B(\s _K)$ by the map $N$ encountered in
  Theorem~\ref{thm:oszint}, our assumption on $Y$ determines the first
  two of the three homologies of the first claim.  For the computation
  of the third homology we invoke the fact that for $p$ large enough
  the surgery provides an $L$-space, hence $H_*(A_i(\s ))$ cannot
  contain any $U$-torsion submodules.  Indeed, by
  Lemma~\ref{lem:CompletionLemma} a $U$-torsion submodule in $H_*(A_i
  (\s ))$ implies the existence of a $U$-torsion submodule in $H_*
  ({\mathbf {A}}_i (\s ))$, but such submodules are in $\ker v_*$ and
  $\ker h_*$, eventually giving rise to nontrivial $U$-torsion in the
  homology of the mapping cone $MC({\mathbb {A}}_{\t }, {\mathbb
    {B}}_{\t }, f)$.  This, however, contradicts the assumption that 
  $Y_p(K)$ is an $L$-space (implying, by
  Lemma~\ref{lem:CompletionLemma}, that $\HHFm (Y_p(K), \t )=\Field
  [[U]]$ for all spin$^c$ structure).

The fact that $H_*(A_i(\s ))$ is equal to $\Field [U]$ then follows
from the fact that the embedding $v$ induces an isomorphism on
homologies with $\Field [\uuinv ]$-coefficients, and $H_* (\CFinf (Y,
\s ))=\Field [\uuinv ]$: it shows that $H_* (A_i(\s )\otimes _{\Field
  [U]}\Field [\uuinv ])=\Field [\uuinv ]$, implying $H_*(A_i (\s
))=\Field [U]$. This observation concludes the proof of (1).

  For (2) consider now the two maps $\psi = \psi _i(\s ) \colon A_i(\s )\to
  A_{i+1}(\s )$ and $\phi = \phi _{i+1} (\s ) \colon A_{i+1}(\s ) \to
  A_i(\s )$, where the latter is the natural embedding after
  multiplication by $U$. Obviously the composition is simply
  multiplication by $U$, hence it induces the same map on homologies.
  Therefore the induced maps $\psi _*$ and  $ \phi _*$ are either
  multiplications by 1 or by $U$, in such a way that their product is
  multiplication by $U$. This concludes the proof of the claim about
  $d_i (\s )$.

  Considering the maps $v_*$ and $h_*$ over $\Field [\uuinv]$, we
  get isomorphisms, therefore these induced maps are nonzero with
  $\Field [U]$-coefficients.  Since both maps are between two copies
  of $\Field [U]$, the third claim follows.

  By taking the commutative triangle involving $H_*(A_i (\s )),
  H_*(A_{i+1}(\s ))$ and $H_*(B(\s ))$, the claimed change of $a_i(\s
  )$ immediately follows. Similarly, the commutative square involving
  $H_*(A_i (\s )), H_*(A_{i+1}(\s )), H_*(C_i (\s ))$ and
  $H_*(C_{i+1}(\s ))$, together with the multiplication by $U$ map
  from $C_{i+1}(\s )$ to $C_i(\s )$ (inducing an isomorphism on
  homology) verifies the claim about $b_i(\s )$.

  Finally, if the monotone sequences $\{ a_i(\s )\}, \{ b_{-j} (\s
  )\}$ stabilize on any other value, Theorem~\ref{thm:oszint} would
  produce homologies which are not finitely generated $\Field
  [[U]]$-modules, contradicting basic properties of Heegaard Floer
  homology groups. This observation proves (5) and concludes the proof
  of the theorem.
\end{proof}

Using information about the results of surgeries on $K$, and by
applying Theorem~\ref{thm:oszint} we get additional information about
the exponents $a_i(\s ), b_i(\s )$ and $d_i(\s )$.

\begin{lem}\label{lem:minimumok}
  The quantity $\min \{ a_i (\s ) , b_i(\s ) \}$ can be determined
  from the Heegaard Floer homology of $Y_n(K)$ for $n \in \Z $ large
  enough in absolute value and from a suitably chosen spin$^c$
  structure on it.
\end{lem}
\begin{proof}
  Indeed, for $\vert n\vert $ large enough and an appropriate choice
  of the spin$^c$ structure $\t$ on $Y_{n}(K)$, the mapping cone
  computing $\HHFm (Y_n(K), \t)$ involves the maps $v_i\colon {\mathbf
    {A}}_i (\s )\to {\mathbf {B}}(\s )$ and $h_i \colon {\mathbf
    {A}}_i (\s )\to {\mathbf {C}}_i (\s )$, together with some other
  similar maps, for which one of the exponents $a_{i+\alpha\cdot n}$
  or $b_{i+\alpha \cdot n}$ is equal to 0. Therefore those parts can
  be contracted when computing the homology, and we are left with the
  complex $g\colon \Field [[U]]\to \Field [[U]]\oplus \Field [[U]]$,
  where $g$ maps the generator of the domain into $(U^{a_i(\s )} ,
  U^{b_i(\s )})$. The homology of this complex is isomorphic to
  $\Field [[U]]\oplus \Field [U]/(U^m)$ with $m =\min \{ a_i (s), b_i(\s
  )\}$, hence the $\Field [[U]]$-module structure of the resulting
  homology recovers $\min \{ a_i (\s ), b_i (\s )\}$.
\end{proof}

There is a further property of the mapping cone considered in
Theorem~\ref{thm:oszint}, which we state presently and will use in our
subsequent argument.  Notice that in the notation of
Theorem~\ref{thm:oszint}, there is a natural map of the chain complex
${\mathbf {B}}(\s _{n[K]})$ (computing the homology group $\HHFm (Y,
\s _{n[K]})$ into the mapping cone $MC({\mathbb {A}}_{\t }, {\mathbb
  {B}}_{\t }, f)$ by embedding ${\mathbf {B}}(\s _{n[K]})$ into the
subcomplex ${\mathbb {B}}_{\t }$ of the mapping cone.  The resulting
composition therefore gives a map $\varphi \colon {\mathbf {B}}(\s
_{n[K]})\to MC({\mathbb {A}}_{\t }, {\mathbb {B}}_{\t }, f)$, inducing
a map $\varphi _*\colon \HHFm (Y, \s _{n[K]})\to \HHFm (Y_p(K), \t )$.
For the next statement recall that a 4-dimensional spin$^c$ cobordism
$(W, \u )$ between the spin$^c$ 3-manifolds $(Y_1, \s _1)$ and $(Y_2,
\s _2 )$ induces a map $F_{W, \u }\colon \HFm (Y_1, \s _1) \to \HFm
(Y_2, \s _2)$ and a corresponding map ${\mathbf {F}}_{W,\u }\colon
\HHFm (Y_1, \s _1) \to \HHFm (Y_2, \s _2)$ on the completed theories.
This map is defined on the chain complex level by appropriately
counting certain holomorphic \emph{triangles} in the symmetric product
${\rm {Sym}} ^g (\Sigma )$, see \cite{OSz4manifold}.

\begin{prop}(\cite{OSzint, OSzrac})
The map $\varphi _* $ described above is equal to the map ${\mathbf
  {F}}Â§_{W_p(K),\u }$ induced on the Heegaard Floer homology groups by
the spin$^c$ cobordism $(W_p(K), \u )$ where
\begin{itemize}
\item the 4-dimensional cobordism $W_p(K)$ is defined by the surgery
  (viewed as a 4-dimensional 2-handle attachment),  and 
\item the spin$^c$ structure $\u$ is given by the property that it
  restricts as $\s_{n[K]}$ and $\t$ on the boundary components $Y$
  and $Y_p(K)$ of $W_p(K)$, and its first Chern class takes the value
  $i+\alpha \cdot n $ on the generator of $H_2(W_p(K); \Z )$. \qed
\end{itemize}
\end{prop}
As a map induced by a 4-dimensional spin$^c$ cobordism, $\varphi _*$
has a well-defined degree shift. (Recall that since $Y$ and $Y_p(K)$
are assumed to be rational homology spheres, all spin$^c$ structures
on them are torsion, and therefore the Heegaard Floer homology groups
admit absolute Maslov gradings.) By \cite{OSzabs} this quantity is
equal to
\begin{equation}\label{eq:degreeshift}
\frac{1}{4}(c_1^2(\u)-3\sigma (W_p(K))-2\chi (W_p(K)).
\end{equation}
This expression then provides us a way to control the difference
$a_i(\s )-b_i(\s )$. Indeed, suppose that $g$ generates the homology
$H_*({\mathbf {A}}_i(\s ))$, and it has Maslov grading $\mu (g)$. Then
the boundary map in the mapping cone (pointing to ${\mathbf {B}}(\s
)$) maps it to an element of Maslov grading $\mu (g)-1$.  This element
(in $H_*({\mathbf {B}}(\s ))$ is equal $U^{a_i(\s)}$-times the
generator of $H_* ({\mathbf {B}}(\s ))$. The generator of
$H_*({\mathbf {B}}(\s ))$, however, in the mapping cone has Maslov
grading
\[
d(Y, \s )+ \frac{1}{4}(c_1^2(\u)-3\sigma (W_p(K))-2\chi (W_p(K)),
\]
since in $\HHFm (Y, \s )$ the Maslov grading of the generator is by
definition equal to $d (Y, \s )$ and the map ${\mathbf {F}}_{W_p(K),
  \u }$ shifts degree by the quantity given in
Equation~\eqref{eq:degreeshift}. In conclusion, for the Maslov grading
$\mu (g)$ of the generator $g\in H_*({\mathbf {A}}_i (\s ))$ we get
that
\begin{equation}\label{eq:Mas1}
  \mu (g)=-2a_i(\s )+d(Y, \s )+ \frac{1}{4}(c_1^2(\u)-
  3\sigma (W_p(K))-2\chi (W_p(K))-1.
\end{equation}
The same argument, now applied to the map ${\mathbf {A}}_i(\s )\to
{\mathbf {B}}(\s _{[K]})$ gives
\begin{equation}\label{eq:Mas2}
  \mu (g)=-2b_i(\s )+d(Y, \s _{[K]})+ \frac{1}{4}(c_1^2(\u ')-
  3\sigma (W_p(K))-2\chi (W_p(K))-1.
\end{equation}
Subtracting Equation~\eqref{eq:Mas1} from Equation~\eqref{eq:Mas2} we
arrive to the following conclusion:

\begin{thm}\label{thm:kulonbseg}
The difference $a_i(\s )-b_i (\s )$ is equal to
\[
\frac{1}{2}(d(Y, \s ) - d(Y, \s _{[K]}))+\frac{1}{8}(c_1^2(\u )-c_1^2 (\u '))
\]
where the spin$^c$ structures $\u , \u '$ on the cobordism $W_p(K)$
are determined by the spin$^c$ structures $\s$ and $\s _{[K]}$ on $Y$,
by $\t$ on $Y_p(K)$ and by the values $i$ and $i+\alpha$ of $c_1(\u ),
c_1 (\u ')$ on the generator of $H_2 (W_p(K); \Z )$. \qed
\end{thm}

\subsection{The connected sum formula}
\label{ss:consummhf}
By taking the connected sum of the pairs $(Y_1, K_1)$ and $(Y_2, K_2)$
in the points of $K_1$ and $K_2$, we get the connected sum $(Y_1\#
Y_2, K_1 \# K_2)$. As it was shown in \cite[Theorem~5.1]{OSzrac}, the
chain complex $\CFm (Y_1\# Y_2, \s _1\# \s _2)$ is the tensor product
of the chain complexes $\CFm (Y_1, \s _1)$ and $\CFm (Y_2, \s _2)$
(over $\Field [U]$), and the Alexander grading $A_{\#}$ on the
connected sum is simply the sum of the Alexander gradings of the
individual Alexander gradings $A_1, A_2$:
\[
A_{\#}(x\otimes y)=A_1(x)+A_2 (y).
\]
In a similar manner, the doubly filtered chain complexes of $(Y_i,
K_i)$ (for $i=1,2$) determine the doubly filtered chain complexes of
the connected sum $(Y_1\# Y_2, K_1 \# K_2)$.

\section{Lattice homology}
\label{sec:lattice}

With some modifications, the results proved in the previous section
can be verified in the lattice homology setting as well. We go through
the statements and arguments below.

We start by a short recollection of lattice homology and knot lattice
homology. (For more details, see \cite{lattice, latticeknot}).
Lattice homology was introduced by N\'emethi \cite{lattice}, and it
associates an algebraic invariant to a negative definite plumbing
graph.  For simplicity suppose that $G$ is a negative definite
plumbing tree/forest on the vertex set $V=\Vertices (G)$, giving
rise to the plumbing 4-manifold $X_G$ and its boundary 3-manifold
$Y_G$. Then the $\Field [\uuinv]$-module $\CFinfComb (G)$, freely
generated by the pairs $[K,E]$ with $K\in H^2 (X_G; \Z )$
characteristic and $E\subset V$ admits a natural $j$-filtration (by
the negative of the exponent of $U$) and a boundary map $\partial$
such that the subcomplex $\CFmComb (G)=\{ x\in \CFinfComb (G) \mid
j(x)\leq 0\}$ computes the \emph{lattice homology} $\HFmComb (G)$ of
$G$. (For the definition of the boundary map see \cite{lattice,
  latticeknot}.) The chain complex splits according to the spin$^c$
structures of $G$ (or, equivalently, of $Y_G$), and the homology
$\HFmComb (G, \s )$ of $\CFmComb (G, \s )$ for a spin$^c$ structure
$\s \in Spin ^c (Y_G)$ further splits according to the
\emph{$\delta$-grading} of $\CFmComb (G, \s ) $ given for a generator
$[K,E]$ by the cardinality $\vert E\vert $ of $E$:
\[
\HFmComb (G, \s )=\oplus _{i=0}^{\vert V \vert} \HFmComb _i (G, \s ).
\]
The chain complex comes with an absolute Maslov grading which descends
to a Maslov grading on the homologies, having the property that
multiplication by $U$ is of degree $-2$. 
The formula for this grading gr for a generator $[K,E]$ 
reads as follows:
\begin{equation}\label{eq:maslov}
{\rm {gr}}([K,E])=2g[K,E]+\vert E \vert +\frac{1}{4}(K^2+\vert V\vert ),
\end{equation}
where $g[K,E]=\min \{ \sum _{u\in I }K(u)+(\sum _{u\in I}u)^2\mid I\subset E\}$.
(Recall that since $Y_G$ is a rational homology sphere, 
the square $K^2$ of $K\in H^2 (X_G; \bfq )$ is well-defined and 
$K^2 \in \bfq $.)

The algebraic structure of $\HFmComb (G, \s )$ is similar to that of
the Heegaard Floer group of $(Y_G, \s)$: it is a finitely generated
$\Field [U]$-module, which is a direct sum of a free part isomorphic
to $\Field [U]$, and a $U$-torsion part. The Maslov grading of the
generator of the free part is the $d$-invariant $d^L(G, \s )$ in
lattice homology.

The significance of rationality and almost rationality of a graph $G$
in the present context comes from the following theorem of N\'emethi:
\begin{thm}(N\'emethi, \cite{lattice})\label{thm:nemethi}
Suppose that $G$ is a negative definite plumbing tree.
\begin{itemize}
\item If $G$ is a rational graph then $\HFmComb (G, \s )=\HFmComb
  _0(G, \s )= \Field [U]$ for every spin$^c$ structure $\s$.
\item If $G$ is almost rational, then $\HFmComb (G, \s )=\HFmComb
  _0(G, \s )$, that is, the homology is supported in the lowest
  $\delta$-grading.
\item For an almost rational graph $G$ and spin$^c$ structure $\s$ the
  lattice homology $\HFmComb (G, \s )$ is isomorphic (as a Maslov
  graded $\Field [U]$-module) to $\HFm (Y_G, \s )$. \qed
\end{itemize}
\end{thm}
\begin{rem}\label{rem:teszt}
  In fact, according to \cite{lattice} the property that $\HFmComb (G,
  \s )= \Field [U]$ for every spin$^c$ structure does characterize
  rational graphs.  We will not use this direction of N\'emethi's
  result in our subsequent discussions.

  There is a fairly simple combinatorial algorithm due to Laufer
  \cite{Laufer}, which decides whether a negative definite tree is
  rational or not. The algorithm proceeds as follows.  Consider
  $Z_1=\sum E_i$ and determine all values of $Z_1\cdot E_i$. If there
  is a product which is at least 2, then the algorithm stops and the
  graph is not rational.  If all the products satisfy $Z_1\cdot E_i
  \leq 0$ then the algorithm stops again, and the graph is
  rational. If there is an index $i$ such that $Z_1\cdot E_i=1$, then
  define $Z_2=Z_1+E_i$ and repeat the previous step with $Z_2$ in the
  role of $Z_1$. Iterating the above procedure we get a sequence of
  vectors $Z_1, Z_2, \ldots $.  According to \cite{Laufer} this procedure stops after finitely
  many steps, and hence determines whether the graph is rational or
  not.
\end{rem}

Suppose now that $\Gamma _{v_0}$ is a tree with a distinguished vertex
$v_0$, and with framings on all the other vertices.  Assume that the
plumbing tree/forest $G=\Gamma _{v_0}-v_0$ is a negative definite
graph. As it is discussed in \cite{latticeknot}, in this situation $v_0$
induces a filtration, the \emph{Alexander filtration}, on $\CFmComb
(G, \s )$ and on $\CFinfComb (G, \s)$, turning the latter into a
doubly filtered chain complex (exactly as we saw it in the Heegaard
Floer setup).

As in the previous section, the doubly filtered chain complex allows
us to define various subcomplexes.  Fix a spin$^c$ structure $\s \in
Spin^c(Y_G)$. In order to keep track which theory we are in, we will
add a superscript $L$ referring to \emph{lattice} homology. Consider
therefore the subcomplexes
\[
B^L(\s ) = \{ x\in \CFinfComb (G, \s )\mid j(x)\leq 0\},
\]
\[
C_i^L(\s ) =\{ x\in \CFinfComb (G, \s )\mid A(x)\leq i\} ,
\]
\[
A_i^L(\s )=C_i^L(\s )\cap B^L(\s ) =\{ x\in \CFinfComb (G, \s )\mid
j(x)\leq 0, A(x)\leq i\}.
\]
(Once again $i\equiv i_{\s }\pmod 1$, where 
$i_{\s}\in [0,1)\cap \bfq$ is a rational number attached to 
the spin$^c$ 3-manifold $(Y_G, \s )$.)
The important maps are again the natural embeddings $\psi _i ^L(\s )
\colon A_i ^L(\s )\to A_{i+1}^L(\s )$, $v_i^L(\s ) \colon A^L_i(\s
)\to B^L(\s )$ and $h_i^L(\s )\colon A^L_i (\s )\to C^L_i (\s )$.  As
before, $\phi _i^L(\s )\colon A_{i+1}^L (\s )\to A_i ^L (\s )$ is
defined as the composition of multiplication by $U$ and the natural
embedding.  Following similar notations in the Heegaard Floer context,
the spin$^c$ structure given by twisting $\s$ with $n[v_0]$ will be
denoted by $\s _{nv_0}$.

\begin{thm}(\cite{latticeknot}) \label{thm:latticemappingcone} 
  Let $G_{-k}(v_0)$ denote the graph we get from $\Gamma _{v_0}$ by
  attaching framing $-k$ to $v_0$ with $k \in \N$ in such a manner
  that the resulting plumbing graph is negative definite.  Fix a
  spin$^c$ structure $\t$ on $G_{-k}(v_0)$. Then there are rational
  numbers $i$ and $\alpha$, a spin$^c$ structure $\s \in Spin^c (G)$
  (depending on $k$ and $\t$) and chain maps $N_n^L \colon
  C^L_{i+\alpha \cdot n}(\s _{nv_0})\to B^L (\s _{(n+1)v_0})$ such
  that the chain complexes
\[
{\mathbb {A}}^L_{\t }=\oplus _{n=-\infty }^{\infty} A^L_{i+\alpha \cdot n}(\s _{nv_0})
\]
\[
{\mathbb {B}}_{\t }^L=\oplus_{n=-\infty }^{\infty} B^L (\s _{nv_0})
\]
and the map $f^L\colon {\mathbb {A}}^L_{\t } \to {\mathbb {B}}^L_{\t
}$ defined on the component $A^L_{i+\alpha \cdot n}(\s _{nv_0})$ as
$v_{i+\alpha \cdot n}^L+N^L_n\circ h^L_{i+\alpha \cdot n}$ provide a
mapping cone $MC({\mathbb {A}}^L_{\t }, {\mathbb {B}}^L_{\t }, f^L)$
with the property that its homology is isomorphic to $\HFmComb
(G_{-k}(v_0), \t )$.  Moreover, the lowest $\delta $-grading $\HFmComb
_0 (G_{-k}(v_0), \t )$ is supported in the homology of the subcomplex
$ {\mathbb {B}}_{\t }^L$.  \qed
\end{thm}
\begin{rem}
Since we assume both $G$ and $G_{-k}(v_0)$ to be negative definite, in
the above theorem we did not need to consider the 
completed version of the theory.
\end{rem}

The property in the Heegaard Floer discussion of requiring 
$L$-space surgery along the knot $K\subset Y$ is now substituted with 
the following observation. 

\begin{lem}\label{lem:arat}
  Suppose that $\Gamma _{v_0}$ has the property that the plumbing
  graph $G=\Gamma _{v_0}-v_0$ is rational and connected (i.e. $v_0$ is
  a leaf). If we equip $v_0$ with framing $-k$ with $k\in \bfn $ large
  enough, then the resulting plumbing tree $G_{-k}(v_0)$ is almost
  rational.
\end{lem}
\begin{proof}
  Since $v_0$ is a leaf, it is connected to a single vertex $w$.  It
  is not hard to see that by decreasing the framing on $w$
  sufficiently, we get a rational graph, concluding the
  proof. (Cf. the algorithm in Remark~\ref{rem:teszt} about verifying
  that a graph is rational.)
\end{proof}

The following lemma will be useful in the comparison with Heegaard
Floer homology:

\begin{lem}\label{lem:possurg}
  Suppose that $\Gamma _{v_0}$ has the property that the plumbing
  graph $G=\Gamma _{v_0}-v_0$ is rational and connected (i.e. $v_0$ is
  a leaf). Then, if we equip $v_0$ with a framing $p$ for $p\in\bfn$ 
  large enough, the three-manifold $Y_{G_p(v_0)}$ is an $L$-space.
\end{lem}

\begin{proof}
  Let $G'$ be the graph obtained by decreasing the framing on $w$ by
  one, and replacing the edge connecting $w$ and $v_0$ by a string of
  edges with $p-1$ new vertices, all labelled with $-2$. By simple
  Kirby calculus, $Y_{G_p(v_0)}\cong Y_{G'}$. Applying Laufer's
  algorithm (Remark~\ref{rem:teszt}), it is easy to see that $G'$ is a
  rational graph. The lemma now follows from Theorem~\ref{thm:nemethi}.
\end{proof}

Recall that by Theorem~\ref{thm:nemethi} the property of being almost
rational implies that the lattice homology of any graph $G_{-k}(v_0)$
with sufficiently large $k\in \N$ is supported in $\HFmComb
_0(G_{-k}(v_0))$. Assume now that $\Gamma _{v_0}$ has the special
property that $v_0$ is a leaf and $G=\Gamma _{v_0}-v_0$ is a rational
graph. The structure theorem in this setting has a similar shape as it
was described in Theorem~\ref{thm:aibidi} in the Heegaard Floer
context (although the proof  is slightly different).

\begin{thm}\label{thm:aibidiL}
  Suppose that $\Gamma _{v_0}$ is a given tree with one distiguished
  vertex $v_0$ such that the plumbing graph $G=\Gamma _{v_0}-v_0$ is a
  connected, rational graph. Then
\begin{enumerate}
\item $H_*(B^L(\s ))=H_* (C^L_i(\s ))=H_*(A^L_i(\s ))=\Field [U] $ for
  every spin$^c$ structure $\s$ on $G$.
\item The map $\psi ^L _i(\s)_*$ induced by the embedding of $A_i
  ^L(\s )$ into $A_{i+1}^L (\s )$ on the homologies is multiplication
  by $U^{d_i^L(\s )}$ with $d_i^L(\s )\in \{ 0, 1\}$.
\item The maps $v_i^L (\s )_*$ and $h_i^L(\s)_*$ are multiplications
  by $U^{a_i^L(\s)}$ and $U^{b_i^L(\s)}$, respectively, where
  $a_i^L(\s ), b_i^L(\s )$ are nonnegative integers.
\item If $d_i^L(\s )=0$ then $a_i^L(\s)=a_{i+1}^L(\s)$ and $b_i^L(\s
  )=b_{i+1}^L(\s )-1$, while if $d_i^L(\s )=1$ then
  $a_i^L(\s)=a_{i+1}^L(\s )+1$ and $b_i^L(\s )=b_{i+1}^L(\s )$.
\item Finally, for all $i,j$ large enough $a_{i}^L(\s )=0$ and
  $b_{-j}^L(\s )=0$.
\end{enumerate}
\end{thm}

\begin{proof}
  By definition $B^L(\s )$ is chain homotopic to $\CFmComb (G, \s )$,
  hence by our assumption on $G$ it follows from
  Theorem~\ref{thm:nemethi} that $H_* (B^L(\s ))=\Field [U]$. The map
  $N$ encountered in Theorem~\ref{thm:latticemappingcone} (multiplied
  with the appropriate $U$-power) provides a chain isomorphism between
  $C_i^L(\s )$ and $B^L(\s _{v_0})$, hence the assumption on $G$ also
  shows that $H_* (C_i^L (\s ))=\Field [U]$.

  The argument for $A_i^L(\s )$ is slightly more complicated.  Any
  $U$-torsion of $H_*(A_i ^L(\s ))$ must be in the kernel of $v_*$ and
  $h_*$, since their images are in free $\Field [U]$-modules. From
  these maps the homology $\HFmComb (G_{-k}(v_0))$ can be computed for
  any $k$. By Theorem~\ref{thm:nemethi} and Lemma~\ref{lem:arat} these
  homologies are concentrated on level 0, while the kernel of $v_*^L$
  and $h_*^L$ determine nontrivial subgroups of $\HFmComb _{>0}$. From
  this observation we conclude that $A_i^L(\s )$ is a free
  $\Field[U]$-module.  The rank over $\Field [U]$ can be determined
  using the Universal Coefficient Formula, and the fact that $v_*^L$
  provides isomorphisms between $A_i^L(\s )\otimes _{\Field [U]}\Field
  [\uuinv]$ and $B^L(\s )\otimes _{\Field [U]}\Field [\uuinv]$ (and
  $\HFinfComb (G, \s )=\Field [\uuinv]$ for every negative definite
  tree $G$). Therefore $H_*(A_i(\s ))$ is a free $\Field [U]$-module
  with $H_*(A_i (\s )\otimes _{\Field [U]}\Field [\uuinv ])=\Field
  [\uuinv ]$, implying that $H_* (A_i (\s ))=\Field [U]$, concluding
  the proof of (1).

  The nontriviality of the induced maps with $\Field [\uuinv
  ]$-coefficients show that the $U$-equivariant maps $v_*^L$ and
  $h_*^L$ are both nontrivial, and hence are multiplications by some
  $U$-power $U^{a_i^L(\s )}$ and $U^{b_i^L(\s )}$. The rest of the
  proof of the theorem now proceeds exactly as the proof of
  Theorem~\ref{thm:aibidi}.
\end{proof}

Since knot lattice homology and lattice homology are connected by the same
mapping cone description which connects knot Floer homology and Heegaard Floer
homology, the proof of 
Lemma~\ref{lem:minimumok} applies verbatim in the present 
context, giving:

\begin{lem}\label{lem:minimumoklattice}
  The quantity $\min \{ a_i^L (\s ) , b_i^L (\s ) \}$ can be
  determined from the lattice homology of $G_{-k}(v_0)$ for $k \in \N
  $ large enough and from a suitably chosen spin$^c$ structure $\t $
  on it. \qed
\end{lem}
 
Similar considerations as in the Heegaard Floer context allow us to get
information about the difference $a_i(\s )-b_i(\s )$. As before, in
these arguments we will rely on Maslov gradings.  Since no maps
induced by cobordisms have been introduced for lattice homologies (but
such maps play an important role in the proof of
Theorem~\ref{thm:kulonbseg}), we give a slightly modified argument.
Below we spell out the details of the
lattice homological counterpart of Theorem~\ref{thm:kulonbseg}.

Recall that the lattice chain complex $\CFmComb (Y, \s )$ admits a
Maslov grading. In the mapping cone construction of
Theorem~\ref{thm:latticemappingcone} the chain complexes $A_i (\s
_{nv_0})$ and $B(\s _{nv_0})$ get Maslov gradings from the mapping
cone. For these gradings the maps $v,h$ (as boundary maps) are of
degree $-1$.  On the other hand, $B(\s _{nv_0})$ is isomorphic as a
chain complex to $\CFmComb (G, \s _{nv_0})$, where the isomorphism is
given in \cite[Proposition~5.5]{latticeknot}. This isomorphism sends a
generator $[L,H]$ in $\CFmComb (G_{-k}(v_0))$ (with the property that
$L\in H^2 (X_G; \Z )$ characteristic and $H\subset \Vertices
(G_{-k}(v_0))=\Vertices (G)\cup \{ v_0\}$ satisfying $v_0\not\in H$) to
$[L\vert _G, H]$. Suppose now that $g$ is a generator of $H_*(A_i(\s
_{nv_0}))=\Field [U]$, and it is of Maslov grading $\mu (g)$.

Since $G$ is a rational graph, we get that the generator of $\HFmComb
(G, \s )$ can be represented by sum of pairs of the form $[K,\emptyset ]$, since
lattice homology in this case is supported in $\delta$-grading
0. According to Equation~\eqref{eq:maslov}, the grading of such an
element is ${\rm {gr}}([K,\emptyset ])=\frac{1}{4}(K^2+\vert
V_G\vert)$.  On the other hand, the generator of $\HFmComb (G, \s )$
is, by definition, of Maslov grading $d^L(G, \s )$.

According to \cite[Proposition~5.5]{latticeknot} the image of
$[K,\emptyset ]$ in the mapping cone is the class represented by
$[L,\emptyset ]$, where $L\vert _G=K$ and $L$ satisfies
\begin{equation}\label{eq:filtje}
\frac{1}{2}(L(\Sigma ) + \Sigma ^2)=i.
\end{equation}
Here $\Vertices (G_{-k}(v_0))=\{ v_0, v_1, \ldots ,v_r\}$ and $\Sigma =
v_0 +\sum _{j>0}a_j \cdot v_j$ with the property that $\Sigma \cdot
v_j =0$ for all $j>0$ and $a_j\in \bfq$. The
expression of \eqref{eq:filtje} uniquely defines $L$ by the
formula
\[
L(v_0)=2i-\Sigma ^2-K(\Sigma -v_0).
\]
Now the Maslov grading of $[L,\emptyset ]$ in the mapping cone is equal
to $\frac{1}{4}(L^2+\vert V_{G_{-k}(v_0)}\vert )$, which is obvioulsy
equal to
\[
\frac{1}{4}(L^2-K^2)+\frac{1}{4}+d^L(G, \s ).
\]
Therefore the Maslov grading $\mu (g)$ for the generator of $H_*
(A_i(\s ))$ is 
\[
\frac{1}{4}(L^2-K^2)+\frac{1}{4}+d^L(G, \s )-2a^L_i(\s ) +1.
\]
A similar argument gives that the same Maslov grading $\mu (g)$ is equal
to 
\[
\frac{1}{4}(L'^2-K'^2)+\frac{1}{4}+d^L(G, \s _{v_0})-2b^L_i(\s )+1,
\]
where $[K', \emptyset ]$ now represents a generator of $\HFmComb (G,
\s _{v_0})$ and $L'$ is its extension.
The difference of these two expressions now provides a formula for
$a_i^L (\s )-b_i^L(\s )$ in terms of $d^L$-invariants of lattice homology groups of 
the background $G$, and squares of certain characteristic cohomology elements:

\begin{thm}\label{thm:kulonbseglattice}
With the notations as above, the difference $a_i^L (\s ) -b_i^L(\s )$
is equal to
\[
\frac{1}{2}(d^L (G, \s )-d^L (G, \s _{v_0}))+\frac{1}{8}( (L^2-K^2) -(L'^2-K'^2)). \qed
\]
\end{thm}

\subsection{The connected sum formula}
\label{ss:consummlattice}
Suppose that $\Gamma _{v_0}$ and $\Gamma '_{v_0'}$ are two given
graphs with distingished vertices $v_0$ and $v_0'$, respectively. By
adding an edge connecting $v_0$ and $v_0'$ to the disjoint union
$\Gamma _{v_0}\cup \Gamma '_{v_o'}$ and then contracting this edge we
get a new graph $\Delta _{v_0=v_0'}$ with the distinguished vertex
$v_0=v_0'$.  The background 3-manifolds $Y_G$ and $Y_{G'}$
corresponding to $\Gamma _{v_0}$ and $\Gamma '_{v_0'}$ determine the
3-manifold $Y$ corresponding to $D=\Delta _{v_0=v_0'} -(v_0=v_0')$ by
$Y_D=Y_G\# Y_{G'}$. Moreover, the knot $K_{v_0=v_0'}$ in $Y_D$
corresponding to the distingushed vertex is isotopic to the connected
sum $K_{v_0}\# K_{v_0'}\subset Y_G\# Y_{G'}=Y_D$.

It was shown in \cite[Section~4]{latticeknot}
that the lattice homology chain complex
$\CFmComb (D, \s\# \s ')$ is the tensor product of $\CFmComb (G, \s )$
and of $\CFmComb (G', \s ')$ over $\Field [U]$, and the Alexander
gradings $A$ and $A'$ add up to produce the Alexander grading induced
by $v_0=v_0'$ on $\CFmComb (D, \s\# \s ')$.  Similarly to the Heegaard
Floer theory situation, the same simple derivation provides the doubly
filtered chain complexes $\CFinfComb (D, \s \# \s ')$.

\section{Comparing the two theories}
\label{sec:compare}

Suppose now that $\Gamma _{v_0}$ is a tree with a distinguished vertex
$v_0$ such that $G=\Gamma _{v_0}-v_0$ is a negative definite connected
rational graph (in particular, $v_0$ is a leaf).  The corresponding
3-manifold $Y_G$ is by Theorem~\ref{thm:nemethi}(1) an $L$-space with
a knot $K=K_{v_0}$ in it. Since by Lemma~\ref{lem:possurg} large
enough surgery on the knot $K$ is also an $L$-space, both
Theorems~\ref{thm:aibidi} and \ref{thm:aibidiL} apply to this
situation. In particular, for a fixed spin$^c$ structure $\s$ on $Y_G$
(or $G$) we get the sequences $\{ a_i (\s )\}, \{ a_i ^L(\s )\}, \{
b_i (\s )\}, \{ b_i ^L(\s )\}$, and $\{ d_i(\s )\}$, $\{ d_i ^L (\s
)\}$.

\begin{prop}\label{prop:dkegyenlok}
  Suppose that $\Gamma _{v_0}$ is a tree with a distinguished vertex
  $v_0$ such that $G=\Gamma _{v_0}-v_0$ is a negative definite
  connected rational graph (in particular, $v_0$ is a leaf). Then, for
  the sequences $\{ a_i (\s )\} =\{ a_i ^L(\s )\}, \{ b_i (\s )\} =\{
  b_i ^L(\s )\}$ hold, and therefore $d_i(\s )=d_i ^L(\s)$ for every $i\in
  i_{\s}+\Z$.
\end{prop}
\begin{proof}
  By Theorems~\ref{thm:aibidi}(4) and \ref{thm:aibidiL}(4) the
  sequences $\{ a_i (\s)\}$ and $\{ a_i ^L(\s )\}$ determine the
  sequences $\{ d_i(\s )\}$, $\{ d_i ^L (\s)\}$ respectively (and
  these sequences then detemine $\{ b_i (\s )\}$ and $\{ b_i ^L (\s
  )\}$), and therefore it is sufficient to verify that $a_i(\s )=a_i
  ^L(\s )$ ($i\in i_{\s }+\Z$) to conclude the statement of the
  proposition.

For the graph $G_{-k}(v_0)$ with $k\in \N$ large enough the lattice
homology group $\HFmComb (G_{-k}(v_0), \t)$ is isomorphic to $\HFm
(Y_{G_{-k}(v_0)}, \t)$: the graph is almost rational by
Lemma~\ref{lem:arat}, hence the isomorphism of the two theories
follows from Theorem~\ref{thm:nemethi} of N\'emethi on almost rational
graphs.  The combination of Lemmas~\ref{lem:minimumok} and
\ref{lem:minimumoklattice}, together with the above isomorphism now
implies that
\begin{equation}\label{eq:minek}
\min \{ a_i (\s ), b_i (\s )\} = \min \{ a_i^L  (\s ), b_i^L(\s )\}.
\end{equation}

In a similar manner, Theorems~\ref{thm:kulonbseg} and
\ref{thm:kulonbseglattice} imply 
that 
\begin{equation}\label{eq:egyenloek}
a_i (\s )-b_i(\s )=a_i ^L (\s )-b_i ^L (\s ).
\end{equation}
Indeed, since $G$ is a rational graph, the lattice homology $\HFmComb
(G, \s )$ and the Heegaard Floer homology $\HFm (Y_G, \s )$ are
isomorphic as Maslov graded $\Field [U]$-modules, and in particular,
the $d$-invariants $d(Y _G, \s )$ and $d^L (G, \s )$ are equal for
every spin$^c$ structure $\s \in Spin ^c (Y_G)$.  Furthermore, the
difference $L^2-K^2$ appearing in Theorem~\ref{thm:kulonbseglattice}
can be identified with $c_1^2(\u )$ of Theorem~\ref{thm:kulonbseg},
since the restriction of $L$ to $G$ is equal to $K$, while on
$G_{-k}(v_0)-G$ the cohomology class $L$ is equal to $c_1(\u )$.
Therefore $c_1^2(\u )$ is equal to $L^2-K^2$, and the same reasoning
shows that $c_1^2(\u ') =L'^2-K'^2$.

These arguments then verify Equation~\eqref{eq:egyenloek}, implying
(with the use of Equation~\eqref{eq:minek}) that $a_i(\s )=a_i ^L (\s
)$, $b_i (\s )= b_i^L (\s )$, and hence $d_i (\s )=d_i ^L(\s )$ for
all $i \in i_{\s }+\Z$, concluding the proof of the proposition.
\end{proof}

\begin{rem}
  The assumption that $v_0$ is a leaf is used only through the fact
  that $G_{-k}(v_0)$ is almost rational (and therefore the proof of
  Theorem~\ref{thm:aibidiL} is applicable). Therefore the above
  identification of $a_i (s)$ with $a_i^L(\s )$ (and consequently $b_i
  (\s )=b_i ^L(\s )$ and $d_i (\s )=d_i ^L (\s )$) follows by the same
  argument provided the framings of $G=\Gamma _{v_0}-v_0$ imply that
  $G_{-k}(v_0)$ is almost rational for $k \in \N $ large enough.
\end{rem}

From this point, a purely algebraic argument about doubly filtered
chain complexes will imply that (at least in the special case when
$v_0$ is a leaf) the doubly filtered chain complexes in the two
theories are filtered chain homotopic. Although the algebraic
considerations here are closely related to the algebraic lemmas
from~\cite{OSzlens}, for the sake of completeness we prefer to give a
self-contained treatment here.

To state the result, we recall some notations from \cite{OSzlens}.
Suppose that $(C, j, A)$ is a doubly filtered chain complex over
$\Field [\uuinv ]$, and denote the associated bigraded complex by the
same triple. (In our application the grading $j$ will be an integer,
while $A$ will be an element of $i_C+\Z$ with some fixed rational
number $i_C\in [0,1)$.) Assume that $(C, j , A)$ has the property that
  $(j(Ux), A (Ux))=(j(x)-1, A(x)-1)$ for a bihomogeneous element $x\in
  C$. In our arguments only special types of chain complexes will
  appear.  The relevant properties are spelled out in the following
  definition.

\begin{defn} 
  \label{def:KnotFloerHomologyType}
  A doubly-filtered complex $(C,j,A)$ over $\Field[\uuinv ]$ with
  Malsov grading $M$ is of {\em knot Floer homology type} if:
  \begin{itemize}
 \item The differential drops the Maslov grading by $1$.
  \item Multiplication by $U$ drops the Maslov grading by $2$.
  \item Multiplication by $U$ drops both $j$ and $A$ by one.
  \item $C$ is filtered chain homotopy equivalent to a finitely
    generated, free chain complex over $\Field[\uuinv ]$.
  \end{itemize}
\end{defn}

View $C$ as an $\Field [U]$-module.  We specify subcomplexes of $C$ by
specifying domains in the plane: For $R\subset \R^2$ the symbol $C(R)$
denotes the subvector space of $C$ generated by those bihomogeneous
elements which have bigradings in $R$. If $R$ satisfies the property
that for $(x_1, y_1)\in R$ and $x_2\leq x_1, y_2\leq y_1$ we also have
$(x_2, y_2)\in R$, then $C(R)$ is a subcomplex of $C$.  The subcomplex
$C\{ (x,y)\in \R ^2 \mid x\leq 0, y\leq i\}$ is denoted by $A_i$.  The
subcomplex $C^-$ is defined to be $C(\{ (x,y)\mid x\leq 0\})$, and we
get the quotient complex ${\widehat {C}}$ of $C^-$ by taking $U=0$ in
the $\Field [U]$-module $C^-$. Since $H_*(C)$ is finitely generated
(as an $\Field [\uuinv ]$-module), it follows that $H_*({\widehat
  {C}})$ is a finite dimensional vector space over $\Field$.

\begin{defn}
  A complex of knot Floer homology type $(C,j,A)$ is said to be 
  of {\em $L$-space type} if
  the homologies of the subcomplexes 
  $$A_i=C(\{ x\in C \mid j(x)\leq 0, \ A(x)\leq i\})$$ are all
  isomorphic to $\Field [U]$.  For such a chain complex the sequences
  $\{ a_i \}, \{ b_i\}$ and $\{ d_i\}$ can be defined exactly as in
  Theorems~\ref{thm:aibidi} and \ref{thm:aibidiL}.
\end{defn}

If $(C,j,A)$ is of $L$-space type, then the embeddings $\psi_i\colon
A_i\to A_{i+1}$ and $\phi_{i+1}\colon A_{i+1}\to A_i$ (we get by
composing the multiplication by $U$ map with the obvious embedding)
induce maps on the homologies which are multiplications by $U^{d_i}$
and $U^{e_{i+1}}$ respectively. Since the composition $\phi_{i+1}\circ
\psi_i\colon A_i \to A_i$ is multiplication by $U$, so $d_i+e_{i+1}=1$
and hence from $d_i, e_{i+1}\geq 0$ if follows that $d_i, e_{i+1}\in \{ 0,1\}$.

Next we define a model complex of $L$-space type.  For that matter fix rational
numbers 
$\{q,\alpha_1,\dots,\alpha_n,\beta_1,\dots,\beta_{n+1}\}$ with
the following properties:
\begin{itemize}
  \item $\alpha_i\equiv \alpha_j\pmod{2\Z}$
    for all $i,j\in 1,\dots n$,
  \item $\beta_i\equiv \beta_j\pmod{2\Z}$
    for all $i,j\in 1,\dots n+1$,
  \item $\alpha_1\equiv \beta_1+1\pmod{2\Z}$, and 
  \item  $\beta_{i}>\alpha_{i}>\beta_{i+1}$ for $i=1,\dots,n$.
\end{itemize}
The model complex of $L$-space type
$C(q,\alpha_1,\dots,\alpha_n;\beta_1,\dots,\beta_{n+1})$ is defined as
follows. The generators (over $\Field[\uuinv ]$) of the complex are
$\{x_1,\dots,x_n\}$ and $\{y_1,\dots,y_{n+1}\}$. The Alexander and
Maslov values of these generators are defined as follows:
\begin{itemize}
  \item for $k=1,\dots,n$, $j(x_k)=0$,
  \item for $k=1,\dots,n$, $A(x_k)=\alpha_k$,
  \item for $k=1,\dots,n$ $M(x_k)=q-2 \sum_{\ell=1}^{k}
    (\beta_\ell-\alpha_\ell)+1;$
  \item for $k=1,\dots,n+1$, $j(y_k)=0$,
  \item for $k=1,\dots,n+1$, $A(y_k)=\beta_k$, and
  \item for $k=1,\dots,n+1$ $M(y_k)=q-2 \sum_{\ell=1}^{k-1}
    (\beta_\ell-\alpha_\ell).$ (In particular, $M(y_1)=q$.)
\end{itemize}
We equip $C(q,\alpha_1,\dots,\alpha_n;\beta_1,\dots,\beta_{n+1})$ with
the differential
\begin{align*}
\partial x_k & = U^{\beta_k-\alpha_k} y_k + y_{k+1} \\
\partial y_\ell & = 0
\end{align*}
for $k=1,\dots,n$ and $\ell=1,\dots,n+1$. It is easy to check that 
\begin{lem}
The complex $C(q,\alpha_1,\dots,\alpha_n;\beta_1,\dots,\beta_{n+1})$ is a chain
  complex of $L$-space type. \qed
\end{lem}

\begin{defn}
  \label{def:Minimal}
  A complex of knot Floer homology type is called {\em minimal} if 
  the differential on its associated graded complex vanishes.
\end{defn}

\begin{lem}
  \label{lem:MinimalModel}
  Every chain complex of knot Floer homology type is filtered chain homotopy equivalent
  to a minimal one.
\end{lem}

\begin{proof}
  We can assume without loss of generality that $C$ is finitely
  generated. Now, the homology of the associated graded object $C'$ is
  a module of free $\Field[\uuinv ]$-modules.  Let $C'$ be the chain
  complex gotten from $C$ by taking the homology of its associated
  graded object, equipped with its induced differential.  It is easy
  to see that $C'$ is chain homotopy equivalent to $C$.
\end{proof}
After these preparations, we are now ready to describe the proposition
which will be the key ingredient in identifying knot Floer homology
and knot lattice homology.

\begin{prop}\label{prop:algebra}
  Let $(C,j,A)$ be a filtered chain complex of $L$-space type.  Then,
  $(C,j,A)$ is homotopy equivalent to one of the model complexes of
  $L$-space type
 $C(q,\alpha_1,\dots,\alpha_n;\beta_1,\dots,\beta_{n+1})$. Indeed, the parameters
  which determine the model are uniquely specified by:
  \begin{itemize}
    \item $2n+1=\#\{i\big| d_i\neq d_{i-1}\}$
    \item Consider the sequence
      $\gamma_1,\dots,\gamma_{2n+1}$, arranged in
      decreasing order, for which
      $d_{\gamma_i}\neq d_{\gamma_i-1}$.  For $i=1,\dots,n+1$,
      define $\beta_i=\gamma_{2i-1}$, and for $i=1,\dots,n$, define
      $\alpha_i=\gamma_{2i}$.
    \item Define $q$ to be the Maslov grading of of the generator of
      $H_*(C(H))=H_*(C^-)\cong\Field[U]$.
  \end{itemize}
\end{prop}

\begin{proof}
  By Lemma~\ref{lem:MinimalModel}, we can assume without loss of
  generality that $C$ is minimal. Our aim is to show that this minimal
  $C$ is equal to the model
  $C(q,\alpha_1,\dots,\alpha_n;\beta_1,\dots,\beta_{n+1})$.

  Let us concentrate on the chain complex $(C, j, A)$.  The group
  at bigrading $(0,i)$ can be determined by an iterated mapping cone;
  i.e. the following $2\times 2$-complex:
  \[
  \begin{CD}
    A_{i+1} @>{\phi_{i+1}}>> A_i \\
    @A{\psi_{i}}AA @AA{\psi_{i-1}}A \\
    A_i @>{\phi_{i}}>> A_{i-1} \\
  \end{CD}
  \]
 There are two natural projection maps from this iterated mapping cone  which we will
 consider. The first  one is the projection to the mapping cone
 $A_i\stackrel{\psi_i}{\longrightarrow} A_{i+1}$. This gives a model
 for the connecting differential from $C(\{j=0,A=i\})\to
 C(\{j=0,A<i\})$.  The other projection maps to $A_i\stackrel{\phi _i
 }{\longrightarrow} A_{i-1}$. This gives a model for the connecting
 differential from $C(\{j=0,A=i\})\to C(\{j<0,A=i\})$.  (Here, again,
 $\phi _i$ is defined as the composition of multiplication by $U$
 followed by the natural embedding.)

  We can take the homology of the $2\times 2$ complex to get a new
  associated graded object:
  \[
  \begin{CD}
    H_*(A_{i+1}) @>{(\phi_{i+1})_*}>> H_*(A_i) \\
    @A{(\psi_{i})_*}AA @AA{(\psi_{i-1})_*}A \\
    H_*(A_i) @>{(\phi_{i})_*}>> H_*(A_{i-1}). \\
  \end{CD}
  \]
 Using the fact that the composites $(\phi_{i+1})_*\circ (\psi_i)_* =
 U$ and $(\psi_{i-1})_*\circ (\phi_i)_*=U$, we can rewrite the above
 square as:
  \[
  \begin{CD}
    H_*(A_{i+1}) \cong \Field[U] @>{U^{1-d_i}}>> \Field[U]\cong H_*(A_i) \\
    @A{U^{d_i}}AA @AA{U^{d_{i-1}}}A \\
    H_*(A_i)\cong\Field[U] @>{U^{1-d_{i-1}}}>> \Field[U]\cong H_*(A_{i-1}). \\
  \end{CD}
  \]
  It is not hard to see that the homology of this square 
  is the trivial group if $d_{i-1}=d_{i}$ and is equal to $\Field $
  if $d_{i-1}\neq d_i$.  Indeed, those results imply that the same
  is true for the group $H_*(C(\{j=0,A=i\})$. 

  Since $d_i, d_{i-1}\in \{ 0, 1\}$, there are two subcases where the
  homology is non-trivial.  The first subcase is where $d_{i-1}=1$ and
  $d_i=0$, in which case the mapping cone looks like:
  \[
  \begin{CD}
    H_*(A_{i+1}) \cong \Field[U]@>{U}>> \Field[U]\cong H_*(A_i) \\
    @A{1}AA @AA{U}A \\
    H_*(A_i)\cong \Field[U] @>{1}>> \Field[U]\cong H_*(A_{i-1}). \\
  \end{CD}
  \]
  In this case, we immediately see that both
  $H_*(C\{j=0,A<i\})=H_*(C\{j<0,A=i\})=0$.  Moreover, we see that the
  homology of $C(\{j\leq 0,A=i\})$ is supported in $j=0$.  This means
  that, under the map in the long exact sequence of homologies induced
  by the short exact sequence 
  $$\begin{CD}    
    0@>>> C(\{j<0,A=i\})@>>>C(\{j\leq 0,A=i\})@>>> C(\{j=0,A=i\})@>>> 0,
  \end{CD}
  $$
  the group $H_*(C(\{j\leq 0,A=i\}))$ maps isomorphically onto
  $H_*(C(\{j=0,A=i\})$. 

  The second subcase is where $d_{i-1}=0$ and $d_i=1$, in which case
  the mapping cone looks like:
  \[
  \begin{CD}
    H_*(A_{i+1}) @>{1}>> H_*(A_i) \\ @A{U}AA @AA{1}A \\ H_*(A_i)
    @>{U}>> H_*(A_{i-1}), \\
  \end{CD}
  \]
  and in this case we have that $ H(C\{j=0,A<i\})=\Field$, and indeed,
  the connecting homomorphism associated to the short exact sequence 
  $$\begin{CD}    
    0@>>> C(\{j=0,A<i\})@>>> C(\{j=0,A\leq i\})@>>> C(\{j=0,A=i\})@>>> 0
  \end{CD}
  $$
 induces an isomorphism $H_* (C(\{j=0,A=i\}))\cong H_*(C(\{j=0,A<i\}))$, and 
 $H_*(C(\{j=0,A\leq i\}))=0$.

Notice first of all that for $i$ large enough the embedding $A_i \to
A_{i+1}$ induces isomorphism on homologies, hence for $i$ large enough
$d_i=0$.  Similarly, for small enough (negative, large in absolute
value) $i$ we get $e_{i+1}=0$ and so $d_i=1$. This shows that whenever
$d_i=0$ and $d_{i-1}=1$, the corresponding $i$ is equal to some $\beta
_k$, while for $i$ with $d_i=0$ and $d_{i-1}$ we get $\alpha _k$.  The
above homological computations then show that $U^{\beta_k-\alpha_k}
y_k$ and $y_{k+1}$ both appear with non-zero multiplicity in $\partial
x_k$, and indeed
  \begin{equation}
    \label{eq:PartDifferential}
    \partial x_k = y_{k+1} + U^{\beta_k -\alpha_k} y_k + L_k,
  \end{equation}
  where $L_k$ is of lower order in the following ways:
  \begin{itemize}
    \item any non-zero term $U^j y_m$ in $L_k$ with $A(U^t
      y^m)=A(x_k)$ has $j>\beta_k-\alpha_k$, and 
    \item any non-zero term $y_m$ (with $j=0$) in $L_k$ has
      $A(y_m)=\beta_m<\beta_{k+1}$.
 \end{itemize}

 Equation~\eqref{eq:PartDifferential} (and the convention that Maslov
 grading drops by $1$ under the differential and by $2$ under
 $U$-multiplication) suffices to identify the Maslov gradings of the
 generators of $C$ with the corresponding generators of the model.
 
 As an $\Field[U]$-module,  $C$ splits into two summands (which are switched
  by the differential): the part in Maslov grading $q\pmod{2}$ (more
  concretely, the part generated over $\Field[U]$ by the $y_k$'s), and
  the part in grading $1+q\pmod{2}$ (i.e., the part generated by the
  $x_k$'s).  We call the first part {\em the part in even parity}, and
  the second {\em the part in odd parity}.

  Equation~\eqref{eq:PartDifferential} immediately implies that
  $\partial$ is injective on the part with odd parity, since by
  Equation~\eqref{eq:PartDifferential} $\partial x_k = y_{k+1}$ plus
  terms with higher $U$-powers in them (i.e. with $j>0$). We claim
  that this property ensures that $\partial y_k=0$. Indeed, observe
  that $\partial y_k$ is a cycle in $C(\{j<0,A<A(y_k)\})$, with odd
  parity. Since the differential is injective on the part with odd
  parity, we conclude that $\partial y_k=0$.

  It remains to verify that $\partial x_k$ has exactly the two stated
  term, i.e. that $L_k=0$. This follows from a more careful look at
  the Maslov gradings. Specifically, according to the grading formulas, it is
  easy to see that if $k<\ell$, then $M(y_k)<M(y_\ell)$ and
  $M(U^{\beta_k} y_k)<M(U^{\beta_\ell} y_\ell)$. Now, consider a
  possible term $U^t y_m$ in $\partial x_k$. If $m>k+1$, then $M(U^t
  y_m)<M(y_m)<M(y_{k+1})=M(x_k)-1$, so the Maslov degree of $U^t y_m$
  is too small for it to appear in $\partial x_k$. If $m<k$, then
  observe that
  $$A(U^t y_m)=\beta_m -t \leq A(x_k)=\alpha_k,$$ which in turn implies
  \begin{align*}
    M(U^t y_m)&=M(U^{t-\beta_m+\alpha_k} U^{\beta_m-\alpha_k} y_m) \\
    & \leq 2(\alpha_k-\beta_m-t) 
    +M(U^{\beta_m-\alpha_k} y_m)\\
    &< M(U^{\beta_k-\alpha_k} y_k)=M(x_k)-1.
  \end{align*}
  Thus, the terms with $m<k$ and with Alexander grading $\leq A(x_k)$
  also have too small Maslov grading to appear in $\partial x_k$.
  This line of reasoning ensures that $\partial x_k$ consists of the
  two stated terms, i.e. $L_k=0$. This completes the identification of
  $(C,j,A)$ with the model complex
  $C(q,\alpha_1,\dots,\alpha_n;\beta_1,\dots,\beta_{n+1})$.
\end{proof}

The combination of Proposition~\ref{prop:dkegyenlok}
and~\ref{prop:algebra} then readily implies the following

\begin{cor}\label{cor:leaf}
  Suppose that $G=\Gamma _{v_0}-v_0$ is a rational tree. If $v_0\in
  \Gamma _{v_0}$ is a leaf then the doubly filtered chain complex
  $(\CFinf (Y_G, \s ),j, A)$ in Heegaard Floer homology and
  $(\CFinfComb (G, \s ), j, A)$ in lattice homology are filtered chain
  homotopic for any spin$^c$ structure $\s \in Spin ^c (Y_G)$. 
\end{cor}
\begin{proof}
By Theorems~\ref{thm:aibidi} and \ref{thm:aibidiL} both the knot Floer
complex of $(Y_G, K_{v_0})$ and the knot lattice complex of $\Gamma
_{v_0}$ are of $L$-space type. According to
Proposition~\ref{prop:dkegyenlok} the corresponding sequences $\{ d_i
(\s )\}$ and $\{ d_i ^L(\s )\}$ are equal, hence the model complexes
determined by these sequences are equal for the two cases. Since by
Proposition~\ref{prop:algebra} the complexes are filtered chain
homotopic to the models described by the statement of the proposition,
the statement of the corollary follows at once.
\end{proof}

\begin{proof}[Proof of Theorem~\ref{thm:main}]
  Corollary~\ref{cor:leaf} verifies the statement of the theorem in
  case $v_0$ is a leaf.  Consider now the case when $v_0$ is not a
  leaf. This means that $\Gamma _{v_0}$ is the connected sum of a
  number of trees/forest, all with distinguished vertices, and in
  which the distinguished vertices are leaves. For those graphs
  Corollary~\ref{cor:leaf} verifies the isomorphism between the two
  theories, and for $\Gamma _{v_0}$ then the connected sum formulae of
  Subsections~\ref{ss:consummhf} and \ref{ss:consummlattice} imply the
  statement of the theorem.
\end{proof} 

\begin{proof}[Proof of Theorem~\ref{thm:egycsucs}]
Consider the graph $\Gamma _w$ we get from $G$ by erasing the framing of
$w$. In order to show that $\HFmComb (G)$ and 
$\HFm (Y_G)$ are isomorphic, it is enough to 
show that 
\begin{itemize}
\item the doubly filtered chain complex in the knot lattice homology
  of the vertex $w$ in $\Gamma _w$, and the doubly filtered chain
  complex in the knot Floer homology of the knot $K=K_w$ in $Y_{\Gamma
    _w-w}$ are filtered chain homotopic, and
\item the maps $N_n$ and $N_n^L$ in the mapping cone constructions of
  the two theories are equal.
\end{itemize}
The first issue is exactly the content of Theorem~\ref{thm:main}.  The
second one follows from the fact that $N_n$ (and similarly $N_n^L$) is
a chain isomorphism between $C_i(\s )$ and $B(\s _{[K]})$ (and between
$C_i ^L(\s )$ and $B^L(\s _{v_0})$, resp.), and since the homologies
of all these chain complexes are isomorphic to $\Field [U]$, the maps
$N_n$ (and similarly $N_n^L$) are uniquely determined, hence are
necessarily equal. This observation then concludes the proof of the 
theorem.
\end{proof}

\end{document}